\documentclass[12pt,reqno]{amsart}
\usepackage[pagebackref=true, colorlinks=true, citecolor=blue]{hyperref}

\usepackage{amssymb,amsmath,graphicx,amsfonts,euscript}
\usepackage{amsmath}
\usepackage{mathtools}

\usepackage{bm}

\usepackage{color}

\usepackage{cleveref}

\newtheorem{theorem}{Theorem}
\newtheorem{remark} {Remark}
\newtheorem{lemma} {Lemma}
\newtheorem{prop}{Proposition}

\newtheorem{claim}{Claim}
\newtheorem{definition}{Definition}

\usepackage[margin=0.9in]{geometry}

\date{\today}

\newcommand{\vare}{\varepsilon}
\newcommand{\K}{\mathcal{K}}

\newcommand{\va}{\varphi}
\newcommand{\blds}{\boldsymbol}
\newcommand{\ds}{\displaystyle}

\DeclareMathOperator{\dv}{div} %
\DeclareMathOperator{\curl}{curl} %
\DeclareMathOperator{\supp}{supp} %

\newcommand{\loc}{\scriptstyle \mathrm{loc}}

\newcommand{\bu}{\blds{u}}
\newcommand{\n}{\blds{n}}

\newcommand{\oo}{\omega}

\usepackage{color}
\usepackage{cases}

\definecolor{Green}{rgb}{0.010,0.7,0.02}

\newcommand{\RR}{\mathbb{R}}
\newcommand{\R}{\mathbb{R}}
\newcommand{\real}{\mathbb{R}}

\keywords{Boundary layers, Navier-Stokes equations, Euler equations, inviscid limit, vorticity maximal function}

\begin{document}
\title[Vorticity measures and the inviscid limit]{Vorticity measures and the inviscid limit}
\author[Constantin, Lopes Filho, Nussenzveig Lopes, Vicol]
{Peter Constantin$^1$, Milton C. Lopes Filho$^2$,\\ Helena J. Nussenzveig Lopes$^2$ and Vlad Vicol$^{3}$}
\address{$^1$ Department of Mathematics, Princeton University, Princeton, NJ 08544, USA}
\address{$^2$ Instituto de Matematica, Universidade Federal do Rio de Janeiro, Caixa Postal 68530, 21941-909, Rio de Janeiro, RJ, Brazil}
\address{$^3$ Courant Institute, New York Univeristy, New York, NY 10012, USA}

\email{const@math.princeton.edu}
\email{mlopes@im.ufrj.br}
\email{hlopes@im.ufrj.br}
\email{vicol@cims.nyu.edu}

\begin{abstract}
We consider a sequence of Leray-Hopf weak solutions of the 2D Navier-Stokes equations on a  bounded domain, in the vanishing viscosity limit. We provide sufficient conditions on the associated vorticity measures, away from the boundary, which ensure that as the viscosity vanishes the sequence converges to a weak solution of the Euler equations. These assumptions are consistent with vortex sheet solutions of the Euler equations.
\end{abstract}

\maketitle
%\tableofcontents

%\section*{Introduction}

The behavior of high Reynolds number flows is a major open problem of nonlinear and statistical physics and of PDE theory. Here we discuss a limited aspect, namely the question whether solutions of the unforced two dimensional Navier-Stokes equations converge weakly on a fixed time interval to solutions of Euler equations in bounded domains. This problem is well understood in the absence of boundaries, in a smooth regime; the answer is then positive, and the convergence holds in strong topologies. The problem is however widely open in general in the presence of boundaries, and the answer is not obvious.  Boundary layers exist, and their limiting behavior is poorly understood. In this paper we follow up on a result obtained in \cite{CV2018} by the first and fourth author for two dimensional flows. We extend \cite[Theorem 2.1]{CV2018} by weakening the hypotheses; our main result (see Theorem~\ref{mainthm} below) allows us to consider vortex sheet solutions of the ideal fluid equations. We moreover give an explicit example of a vortex sheet limit Euler solution satisfying our weaker hypotheses (See Proposition~\ref{vortsheetexample} below). It is known that if the convergence is assumed in the vanishing {viscosity} limit, then vortex sheets must develop at the boundary~\cite{Kelliher2008}.
In fact, instability of strong shear flows and detachment of the boundary layer suggests that the limiting flow will not be smooth, see \cite{NWKFS2018} for a broad discussion and relevant numerical experiments.
Our result, in contrast with the Kato criterion~\cite{Kato1983}, applies without assuming closeness to a given
smooth solution of the Euler equations, and allows considering weak solutions, such as vortex sheets. The uniform conditions are imposed on the Navier Stokes solutions away from boundaries.

Let $T>0$ and let $\Omega \subset \RR^2$ be a bounded, smooth, connected and simply connected domain. Consider the initial boundary value problem for the incompressible Navier-Stokes equations with viscosity $\nu>0$, given by:
\begin{equation} \label{nuNS}
\left\{
\begin{array}{ll}
\bu_t+(\bu\cdot\nabla)\bu = -\nabla p + \nu \Delta\bu, & \text{ in } (0,T)\times\Omega, \\
\dv\bu = 0, & \text{ in } [0,T) \times \Omega,\\
\bu = 0, & \text{ on } (0,T) \times \partial\Omega,\\
\bu_{|_{t=0}} = \bu_0, & \text{ at } \{0\}\times\Omega.
\end{array}
\right.
\end{equation}

The initial boundary value problem for the incompressible Euler equations corresponds to taking $\nu = 0$ and substituting the {\em no slip} boundary condition $\bu=0 $ by the {\em non-penetration} condition $\bu\cdot\n=0$ on $(0,T)\times\partial\Omega$, where $\n$ represents the unit outer normal to $\partial\Omega$.

Let us begin by recalling the definition of a weak solution of the Euler equations.

\begin{definition} \label{weakEuler}
The vector field $\bu \in L^\infty((0,T);L^2(\Omega))$ is said to be a {\em weak solution} of the incompressible Euler equations if:
\begin{itemize}
\item  $\dv \bu (t,\cdot) = 0$, a.e. $t\in (0,T)$, in the sense of distributions, and
\item for each test vector field $\Phi \in C^\infty_c((0,T)\times\Omega)$ such that $\dv \Phi (t,\cdot) = 0$, the following identity holds true:
\begin{equation} \label{wkvel}
\int_0^T \int_\Omega \partial_t \Phi \cdot \bu \, dxdt + \int_0^T\int_\Omega \nabla \Phi : \bu \otimes \bu   \, dxdt = 0.
\end{equation}
\end{itemize}
\end{definition}

Throughout we will use the notation $\nabla^\perp\equiv(-\partial_{x_2},\partial_{x_1})$.

We are now ready to state our main result.

\begin{theorem} \label{mainthm}
Let $\nu_n$ be positive numbers such that $\nu_n\to 0$. Let $\bu^n \in L^\infty((0,T);L^2(\Omega))\cap L^2((0,T);H^1_0(\Omega))$ be a family of Leray-Hopf weak solutions of \eqref{nuNS} with viscosity $\nu=\nu_n$. Set $\oo^n=\oo^n(t,\cdot)=\curl \bu^n \equiv \nabla^\perp \cdot \bu^n(t,\cdot)$.

Assume the following:
\begin{enumerate}
\item there exists $\bu^\infty$ such that $\bu^n \rightharpoonup \bu^\infty$ weak-$\ast$ $L^\infty(0,T;L^2(\Omega))$;
\item $\{\oo^n\} \subset L^\infty((0,T);L^1_{\loc}(\Omega))$ and, for each $\mathcal{K}\subset\subset\Omega$, there exists $C_{\mathcal{K}} >0$ so that
\[\sup_n \sup_{t\in(0,T)} \|\oo^n (t,\cdot)\|_{L^1(\mathcal{K})} \leq  C_{\mathcal{K}}<\infty;\]
\item For any $\K \subset\subset\Omega$ we have
\[\sup_n \int_0^T \left(\sup_{x\in \K}  \int_{B(x;r)\cap\Omega}|\oo^n(t,y)|\,dy \right) \,dt \rightarrow 0 \mbox{ as } r\to 0.\]
\end{enumerate}
Then $\bu^\infty$ is a weak solution of the incompressible Euler equations in the sense of Definition \ref{weakEuler}.
\end{theorem}

\begin{remark}
Let $\{\bu_0^n\}_n \subset L^2(\Omega)$ such that $\|\bu_0^n\|_{L^2} \leq C$. If $\bu^n$ is the unique Leray-Hopf weak solution of \eqref{nuNS} with viscosity $\nu^n$ and initial condition $\bu_0^n$, then $\{\bu^n\}$ is a bounded subset of $L^\infty((0,T);L^2(\Omega))$. Hence, passing to subsequences as needed, Assumption (1) automatically holds true in this case.
\end{remark}

\begin{remark} \label{noinitconds}
There is no mention in Definition \ref{weakEuler} of initial conditions. We observe however that it is easy to incorporate initial data into the weak formulation by taking test vector fields $\Phi \in C^\infty_c([0,T);\Omega)$. Now, if in Theorem \ref{mainthm} the initial data $\bu^n(0,\cdot)\equiv \bu_0^n$ converge weakly in $L^2(\Omega)$ to some $\bu^\infty_0$, then it follows that $\bu^\infty(0,\cdot)=\bu^\infty_0$ in this (new) weak sense as well.
\end{remark}

\begin{remark} \label{tanganddivfree}
We note that, by linearity, $\dv \bu^\infty (t,\cdot) = 0$ in the sense of distributions, a.e. $t \in (0,T)$, since $\dv \bu^n (t,\cdot)=0$. Now, because $\bu^\infty \in L^\infty((0,T);L^2(\Omega))$ is divergence free, its normal component at the boundary has a trace in $L^2(0,T;H^{-1/2}(\partial\Omega))$. Because of weak continuity of the trace operator, and as $\bu^n\in L^2((0,T);H^1_0(\Omega))$, the trace of the normal component of $\bu^\infty$ vanishes on $\partial \Omega$.
\end{remark}

\begin{remark} \label{wildstuff}
We observe in the proof that the vorticity $\oo^\infty = \curl\bu^\infty \equiv \nabla^\perp\cdot\bu^\infty$ belongs to $L^\infty((0,T);\mathcal{BM}_{\loc}(\Omega)\cap H^{-1}(\Omega))$. Moreover, we further show  that $\oo^\infty$ is a weak solution of the {\em vorticity formulation} of the incompressible 2-D Euler equations, in a sense to be made precise, see Definition \ref{intwkvort}. We contrast the solutions $\bu^\infty$ obtained here with {\em wild solutions} of the Euler equation (see for instance the review articles~\cite{CamilloLaszloReview2012,CamilloLaszloReview2017}, and the papers~\cite{BardosEtAl2014,BardosTiti2013} in the case of bounded domains). These wild solutions are also weak solutions in the sense of Definition \ref{weakEuler}, but the corresponding vorticity $\oo^\infty$ is not regular enough to be a weak solution of the vorticity formulation. Also, the {\em wild} weak solutions of the Navier-Stokes equation constructed in~\cite{BuckmasterVicol2017} have vorticity which does lie in $L^\infty(0,T; L^{1+\epsilon} \cap H^{-1+\epsilon})$ for some $\epsilon>0$, and they do converge in the inviscid limit to weak solutions of the Euler equations, but the $L^1$ norm of their vorticity degenerates as the viscosity vanishes (in contrast to Assumption (2) of Theorem~\ref{mainthm}).
\end{remark}

\begin{remark} \label{maxvortdecay}
Assumption (3) is referred to as (time integrated) uniform decay of the vorticity maximal function. Recall that the maximal function of vorticity is defined as
\[M_{\oo(t,\cdot)}(x) \equiv \sup_{r>0} \frac{1}{\pi r^2}\int_{B(x;r)\cap\Omega}|\oo (t,y)|\,dy,\]
so the object being considered in (3) is only reminiscent of the maximal function of vorticity. The terminology ``maximal vorticity function" was used in the   work of  DiPerna and  Majda, see \cite[page 65]{DM1988} and \cite[Theorem 3.1]{DM1987a}, while  studying the weak evolution of vortex sheet initial data. Conditions such as Assumption (3) have appeared previously as non-concentration conditions, for instance, in \cite{Schochet1995}, see also \cite{LNX2001}.
\end{remark}

\begin{remark} \label{assumpts2and3}
We note that if we replace Assumption (3) by Assumption (3'):
\[\sup_n \sup_{0\leq t \leq T} \left(\sup_{x\in \K}  \int_{B(x;r)\cap\Omega}|\oo^n(t,y)|\,dy \right) \,dt \rightarrow 0 \mbox{ as } r\to 0,\]
then Assumption (3') implies  Assumption (2). However, Assumption (3) is more natural in view of the analysis for mirror-symmetric flows, see \cite{LNX2001}.
\end{remark}

\begin{remark}
We emphasize that the assumptions of Theorem~\ref{mainthm} are only posed on compact subdomains $\mathcal K$. The constant $C_{\mathcal K}$ of Assumption (2) and the convergence rate of Assumption (3) are allowed to degenerate as ${\rm dist}(\mathcal K, \partial \Omega) \to 0$. A different practical set of interior sufficient conditions such that $\bu^\infty$ is a weak solution of the Euler equations is provided by~\cite[Theorem 3.1]{CV2018} in 3D and~\cite[Theorem 1]{DrivasNguyen18} in 2D. These are uniform bounds for the interior second order structure function of $\bu^n$, with arbitrarily small exponent,  in a  suitably defined inertial range of scales. These assumptions imply the uniform boundedness of $\oo^n$ in $L^2(0,T;H^{-1+\epsilon_{\mathcal K}}(\mathcal K))$ for some $\epsilon_{\mathcal K}>0$, and thus from the point of view of scaling, the assumptions of Theorem~\ref{mainthm} appear to be more general.

\end{remark}

Before we give the proof of Theorem \ref{mainthm} we  introduce the notion of interior weak solution and then we discuss the equivalence between this notion and the weak solutions as in Definition \ref{weakEuler}.

Denote $G_\Omega=G_\Omega(x,y)$ the Green's function for the Laplacian $\Delta$, with homogeneous Dirichlet boundary conditions on $\Omega$.
%\begin{eqnarray*}
% \Delta_x G_\Omega(x,y) & = & \delta_y(x) \mbox{ in }\Omega,\\
% G_\Omega(x,y) & = & 0 \mbox{ on } \partial\Omega.
%\end{eqnarray*}
We write $G_\Omega[f]$ to denote $\Delta^{-1}_0(f)$, and
we denote the Biot-Savart kernel on $\Omega$ by $K_\Omega=K_\Omega(x,y) \equiv \nabla^\perp_xG_\Omega(x,y)$. %Here $\nabla^\perp=(-\partial_{x_2},\partial_{x_1})$.
We write $K_\Omega[f]$ to denote $\nabla^\perp_xG_\Omega[f]$.

\begin{definition} \label{intwkvort}
The scalar  $\oo \in L^\infty((0,T);\mathcal{BM}_{\loc} (\Omega) \cap H^{-1}(\Omega))$ is said to be an {\em interior weak solution} of the {\em vorticity} formulation of the incompressible Euler equations if:

\noindent for each test function $\varphi \in C^\infty_c((0,T)\times\Omega)$ there exists
$\chi=\chi(x) \in C^\infty_c(\Omega)$, satisfying $0\leq \chi \leq 1$ and $\chi \equiv 1$ in a neighborhood of the support of $\va$, such that the following identity holds true:
\begin{equation}
\label{intwkvortid}
\begin{aligned}
 & \int_0^T\int_{\Omega} \partial_t \va (t,x) \oo (t,x)\, dxdt  + \int_0^T\int_\Omega\int_\Omega H_\Omega^\va  (t,x,y)\chi(x) \oo(t,x)\chi(y) \oo(t,y) \, dxdydt\\
 & \quad \quad + \int_0^T\int_\Omega \int_\Omega K_\Omega(x,y)(1-\chi(y))\chi(x)\cdot\nabla\va(t,x)\oo(t,x)\oo(t,y)\,dxdydt = 0,
\end{aligned}
\end{equation}
where $H_\Omega^\va=H_\Omega^\va(x,y)$ is the {\em auxiliary test function} given by
\begin{equation} \label{auxtestfctn}
H_\Omega^\va(x,y) = \frac{K_\Omega(x,y)\cdot\nabla\va(t,x)+K_\Omega(y,x)\cdot\nabla\va(t,y)}{2}.
\end{equation}
As discussed in Remark~\ref{rem:chi} below, the choice of $\chi$ is not relevant.
\end{definition}

\begin{remark}
We are abusing notation above, as the low regularity of $\oo$ does not allow to write the integrals in identity \eqref{intwkvortid}. However, we remark that all the expressions above make sense when suitably interpreted  (cf.~the discussions after \eqref{Bn} and \eqref{limdeltaA2n} below).
\end{remark}

\begin{remark}
If instead, $\oo \in L^\infty((0,T);\mathcal{BM} (\Omega) \cap H^{-1}(\Omega))$ then the identity \eqref{intwkvortid}  makes sense even if $\chi \in C^\infty(\overline{\Omega})$, $\chi \equiv 1$, giving rise to the usual weak vorticity formulation of the 2D Euler equations, see \cite{Schochet1995} and \cite{ILNInPrep}.
\end{remark}

\begin{lemma} \label{weakvortweakvel}
Let $\oo^\infty \in L^\infty((0,T);\mathcal{BM}_{\loc} (\Omega) \cap H^{-1}(\Omega))$ be an interior weak solution of the vorticity formulation of the incompressible Euler equations. Then $\bu^\infty=K_{\Omega}[\oo^\infty]$ is a weak solution of the Euler equations in the sense of Definition \ref{weakEuler}.

Conversely, let $\bu^\infty \in L^\infty((0,T);L^2(\Omega))$ be a weak solution of the Euler equations in the sense of Definition \ref{weakEuler}. Let $\oo^\infty = \oo^\infty(t,\cdot) = \curl \bu^\infty (t,\cdot) \equiv \nabla^\perp\cdot\bu^\infty (t,\cdot)$, a.e. $t\in (0,T)$, in the sense of distributions. Assume that $\oo^\infty \in L^\infty((0,T);\mathcal{BM}_{\loc}(\Omega))$.
Then $\oo^\infty$ is an interior weak solution in the sense of Definition \ref{intwkvort}.
\end{lemma}

\begin{remark}\label{rem:chi}
In view of Lemma \ref{weakvortweakvel} it follows that, given $\va \in C^\infty_c((0,T) \times \Omega)$,  identity \eqref{intwkvortid} in Definition \ref{intwkvort} is independent of the choice of $\chi \in C^\infty_c(\Omega)$ such that $0\leq \chi \leq 1$ and $\chi \equiv 1$ in a neighborhood of the support of $\va$.
\end{remark}

We postpone the proof of the lemma until after the proof of our main result.

\begin{proof}[Proof of Theorem \ref{mainthm}:]
The strategy of the proof is to use the vorticity equation and pass to the limit in a suitable weak formulation, namely the interior weak vorticity formulation. The proof is concluded once we establish the equivalence between this weak formulation for the vorticity equation and the weak velocity formulation in Definition \ref{weakEuler}, which is the content of Lemma \ref{weakvortweakvel}.

Let $\nu_n \to 0$ and let $\bu^n$ be a solution of the Navier-Stokes equations \eqref{nuNS} with viscosity $\nu_n$, as in the statement of Theorem \ref{mainthm}.
Then $\oo^n = \curl \bu^n$ is a solution of the vorticity formulation of the Navier-Stokes equations:
\begin{equation} \label{vortnuNS}
\partial_t \oo^n + \dv(\bu^n\oo^n) = \nu_n\Delta\oo^n,
\end{equation}
in $(0,T)\times\Omega$.

First, let us note that $\{\oo^n\}$ is bounded in $L^\infty(0,T;H^{-1}(\Omega))$. Next, we observe that we can rewrite the nonlinear term $\dv (\bu^n\oo^n)$ in \eqref{vortnuNS} as second derivatives of terms which are quadratic with respect to the components of $\bu^n$:
\[\dv (\bu^n\oo^n) = \left(\partial_{x_1}^2-\partial_{x_2}^2\right)\bu_1^n\bu_2^n - \partial^2_{x_1 x_2}[(\bu_1^n)^2-(\bu_2^n)^2].\]
It follows that $\{\partial_t\oo^n\}$ is a bounded subset of $L^\infty(0,T;H^{-L}(\Omega))$, for some large $L>0$. Thus, from the Aubin-Lions lemma, we obtain that $\{\oo^n\}$ is a compact subset of $L^\infty(0,T;H^{-M}(\Omega))$, for some $1<M \leq L$. Because we assumed that $\bu^n \rightharpoonup \bu^\infty$ weak-$\ast$ in $L^\infty(0,T;L^2(\Omega))$, it follows by linearity that the accumulation points of $\{\oo^n\}$ are all $\oo^\infty \equiv \nabla^\perp\cdot\bu^\infty$ and hence the whole sequence  $\{\oo^n\}$ converges strongly in $L^\infty(0,T;H^{-M}(\Omega))$, to $\oo^\infty$. Furthermore, clearly we have $\oo^n \rightharpoonup \oo^\infty$ weak-$\ast$ in $L^\infty(0,T;H^{-1}(\Omega))$.

We note that the vector field given by $K_\Omega[\oo^\infty]$ is divergence free, has  $\oo^\infty$ as its two dimensional curl, and is tangent to $\partial\Omega$. Since $\Omega$ was assumed to be  simply connected there is a unique vector field which is divergence free, has curl equal to $\oo^\infty$ and is tangent to the boundary of $\Omega$. Since $\bu^\infty$ satisfies these same conditions, see Remark \ref{tanganddivfree}, it follows that $\bu^\infty = K_\Omega[\oo^\infty]$.

Fix $\va \in C^\infty_c((0,T)\times\Omega)$. Multiplying \eqref{vortnuNS} by $\va$, integrating in $(0,T)\times\Omega$ and transferring derivatives to $\va$ leads to
\begin{equation}\label{firstweakvortnuNS}
\int_0^T\int_\Omega\partial_t\va(t,x) \oo^n(t,x) + \nabla\va(t,x)\cdot \bu^n(t,x) \oo^n(t,x)\,dxdt =
\nu_n\int_0^T\int_\Omega\Delta\va(t,x)\oo^n(t,x)\,dxdt.
\end{equation}

We wish to pass to the limit in each of the terms of \eqref{firstweakvortnuNS}.

The convergence of the linear terms follows easily from the convergence $\oo ^n \rightharpoonup \oo^\infty$ weak-$\ast$ in $L^\infty(0,T;H^{-1}(\Omega))$ and, in particular, in $\mathcal{D}^\prime((0,T)\times\Omega)$. Hence we have
\begin{align}\label{timederiv}
\int_0^T\int_\Omega\partial_t\va(t,x) \oo^n(t,x) \, dxdt & \to \int_0^T\int_\Omega\partial_t\va(t,x) \oo^\infty(t,x) \, dxdt, \text{ and } \\
%\end{equation}
%\begin{equation}
\label{laplacian}
\nu_n\int_0^T\int_\Omega\Delta\va(t,x)\oo^n(t,x)\,dxdt & \to 0,
\end{align}
as $n \to \infty$.

It remains to treat the nonlinear term in \eqref{firstweakvortnuNS}. Let $\chi=\chi(x) \in C^\infty_c(\Omega)$ be a cutoff so that $0\leq \chi \leq 1$ and $\chi \equiv 1$ in a neighborhood of the support of $\va$. In particular, there exists $\eta>0$ such that the supports of $1-\chi$ and of $\va$ are at a distance $\eta$ apart.

We show that
\begin{align} \label{nonlinterm}
\nonumber & \int_0^T\int_\Omega \nabla\va(t,x)\cdot \bu^n(t,x) \oo^n(t,x)\,dxdt \to
\\
& \quad \quad \int_0^T\int_\Omega\int_\Omega H_\Omega^\va  (t,x,y)\chi(x) \oo^\infty(t,x)\chi(y) \oo^\infty(t,y) \, dxdydt \\
\nonumber & \quad \quad + \int_0^T\int_\Omega \int_\Omega K_\Omega(x,y)(1-\chi(y))\chi(x)\cdot\nabla\va(t,x)\oo^\infty(t,x)\oo^\infty(t,y)\,dxdydt
\end{align}
as $n \to \infty$, where $H^\va_\Omega$ was given in \eqref{auxtestfctn}.

Note that because $\bu^n(t,\cdot)$ satisfies the no slip boundary conditions, we have in particular that $\bu^n(t,\cdot)$ can be recovered from $\oo^n(t,\cdot)$ by the Biot-Savart law:
\[\bu^n(t,x)=\int_\Omega K_\Omega(x,y)\oo^n(t,y)\,dy;\]
this holds true in $L^\infty(0,T;L^2(\Omega))$.

We write the nonlinear term as:
\begin{align*}
&\int_0^T\int_\Omega \nabla\va(t,x)\cdot \bu^n(t,x) \oo^n(t,x)\,dxdt \\
&\quad =
\int_0^T\int_\Omega \int_\Omega K_\Omega(x,y)\oo^n(t,y)\cdot\nabla\va(t,x) \oo^n(t,x)\,dydxdt
\\
&\quad=\int_0^T\int_\Omega \int_\Omega K_\Omega(x,y)\oo^n(t,y)\cdot\nabla\va(t,x) \chi(x)\oo^n(t,x)\,dxdydt
\\
&\quad= \int_0^T\int_\Omega \int_\Omega K_\Omega(x,y)\chi(y)\oo^n(t,y)\cdot\nabla\va(t,x) \chi(x)\oo^n(t,x)\,dxdydt
\\
&\qquad +
\int_0^T\int_\Omega \int_\Omega K_\Omega(x,y)(1-\chi(y))\oo^n(t,y)\cdot\nabla\va(t,x) \chi(x)\oo^n(t,x)\,dxdydt
\\
&\quad\equiv A^n + B^n.
\end{align*}
We also introduce
\[A^\infty= \int_0^T\int_\Omega \int_\Omega H^\va_\Omega(x,y)\chi(y)\oo^\infty(t,y) \chi(x)\oo^\infty(t,x)\,dxdydt
\]
and
\[B^\infty = \int_0^T\int_\Omega \int_\Omega K_\Omega(x,y)(1-\chi(y))\oo^\infty(t,y)\cdot\nabla\va(t,x) \chi(x)\oo^\infty(t,x)\,dxdydt.\]

Let us first consider the limit of $B^n$. We note that, if $O_\eta \equiv \{(x,y)\in \Omega\times\Omega \,|\, |x-y|>\eta\}$, then, for a.e. $t\in(0,T)$, the support of the integrand in $B^n$ is contained in $O_\eta$, which avoids the singularity at the diagonal of $K_\Omega(x,y)$. Hence,
\[
B^n = \int_0^T\int_{O_\eta}K_\Omega(x,y)\cdot\nabla\va(t,x) (1-\chi(y))\chi(x)\oo^n(t,y)\oo^n(t,x)\,dxdydt.
\]
Now, for each $t\in (0,T)$, we have $K_\Omega(\cdot,\cdot)\cdot\nabla\va(t,\cdot) \in  C^\infty(O_\eta)$. Moreover, since the support of $\va$ is a compact subset of $(0,T)\times\Omega$, $K_\Omega(x,y)\cdot\nabla\va(t,x)$ vanishes if $x \in \partial \Omega$, for every $t\in(0,T)$. Because $G_\Omega(x,y)$ vanishes for $x\in \Omega$, $y\in \partial\Omega$, it follows that $K_\Omega(x,y)=\nabla^\perp_xG_\Omega(x,y)$ also vanishes for $x\in \Omega$, $y\in \partial\Omega$. Recall that $\oo^n \rightharpoonup \oo^\infty$ weak-$\ast$ in $L^\infty(0,T;H^{-1}(\Omega))$. By linearity, the tensor product $\oo^n \otimes \oo^n$ converges weak-$\ast$, in $L^\infty(0,T;H^{-1}(\Omega\times\Omega))$, to $\oo^\infty\otimes\oo^\infty$. From what we have argued it follows that $K_\Omega(x,y)\cdot\nabla\va(t,x)(1-\chi(y))\chi(x) \in L^1(0,T;H^1_0(\Omega\times\Omega))$, and so we conclude that
\begin{equation} \label{Bn}
B^n \to B^\infty \text{ as } n \to \infty,
\end{equation}
with the spatial integral in $B^\infty$ being interpreted as a duality pairing between $H^{-1}(\Omega \times \Omega)$ and $H^1_0(\Omega \times \Omega)$.

Next we address the convergence of $A^n$. Symmetrizing with respect to the variables $x$ and $y$ as was done in \cite{Schochet1995} for flows in all of $\real^2$, we find
\begin{align*}
A^n
&= \int_0^T\int_\Omega \int_\Omega K_\Omega(x,y)\chi(y)\oo^n(t,y)\cdot\nabla\va(t,x) \chi(x)\oo^n(t,x)\,dxdydt \\
&= \int_0^T\int_\Omega \int_\Omega K_\Omega(y,x)\chi(x)\oo^n(t,x)\cdot\nabla\va(t,y) \chi(y)\oo^n(t,y)\,dydxdt \\
&=\int_0^T\int_\Omega \int_\Omega H^\va_\Omega(x,y)\chi(y)\oo^n(t,y) \chi(x)\oo^n(t,x)\,dxdydt.
\end{align*}
We already know that $\chi\oo^n \to \chi\oo^\infty$ strongly in $L^\infty(0,T;H^{-M}(\Omega))$ and that $\chi\oo^n \rightharpoonup  \chi\oo^\infty$ weak-$\ast$ in $L^\infty(0,T;H^{-1}(\Omega))$. By hypothesis (2), $\{\chi\oo^n\}$ is bounded in $L^\infty(0,T;L^1(\Omega))$, hence, passing to subsequences as needed,  we find $\chi\oo^n$ is weak-$\ast$ convergent in $L^\infty(0,T;\mathcal{BM}(\Omega))$. Putting these facts together allows us to identify the weak-$\ast$-$L^\infty(0,T;\mathcal{BM}(\Omega))$ limit as $\chi\oo^\infty$, so that there is no need to pass to further subsequences; moreover, we have that $\chi\oo^\infty \in L^\infty(0,T;\mathcal{BM}(\Omega))$.

It was established in \cite[Proposition 2.1]{ILNInPrep}, see also \cite[Proposition 2.2]{ILNSu}, that
\[H_\Omega^\va \in L^\infty((0,T)\times\overline{\Omega}\times\overline{\Omega}),\]
analogously to what holds in $\real^2$, see \cite{Schochet1995}. Set $M \equiv \|H_\Omega^\va\|_{L^\infty((0,T)\times\overline{\Omega}\times\overline{\Omega})}.$
In addition, it was also shown in \cite[Proposition 2.1]{ILNInPrep}, see also \cite[Proposition 2.2]{ILNSu}, that $H_\Omega^\va$ is continuous on $\overline\Omega\times\overline\Omega\setminus\{(x,x)\ ;\ x\in \overline\Omega\}$. It is the fact that the diagonal, in $\overline\Omega \times \overline\Omega$, is excluded from the set of continuity of $H_\Omega^\va$ that makes the convergence of $A^n$ above a delicate problem -- we must split the integral in $A^n$ into a portion far from the diagonal and a portion near the diagonal.

Let $\rho \in C^\infty_c(\real)$ be such that  $\rho = 1$ on $[0,1/2]$,  $\supp (\rho) \subset [0,1]$, and $0\leq \rho \leq 1$. For each $\delta > 0$ set $\rho_\delta = \rho_\delta(x,y) = \rho (|x-y|/\delta)$. We rewrite $A^n$ as:
\begin{align*}
A^n &= \int_0^T\int_\Omega \int_\Omega H^\va_\Omega(x,y)\chi(y)\oo^n(t,y) \chi(x)\oo^n(t,x)\rho_\delta(x,y)\,dxdydt
\\
&\qquad  +\int_0^T\int_\Omega \int_\Omega H^\va_\Omega(x,y)\chi(y)\oo^n(t,y) \chi(x)\oo^n(t,x)(1-\rho_\delta(x,y))\,dxdydt \\
& \equiv A^n_1+A^n_2.
\end{align*}

We have:
\begin{equation} \label{A1n}
\sup_n|A^n_1| \leq M\sup_n \|\chi\oo^n\|_{L^\infty(0,T;L^1(\Omega))}\ \sup_n\int_0^T \left(\sup_{x\in\supp(\chi)}\int_{B(x;\delta)}|\chi\oo^n(t,y)|\,dy\right)\,dt.
\end{equation}
It follows from the decay of the vorticity maximal function, Assumption (3), that
\begin{equation} \label{limdeltaA1n}
\lim_{\delta\to 0} \sup_n|A^n_1| = 0.
\end{equation}

In addition,
\begin{equation} \label{A2n}
A^n_2  \to \int_0^T\int_\Omega \int_\Omega H^\va_\Omega(x,y)\chi(y)\oo^\infty(t,y) \chi(x)\oo^\infty(t,x)(1-\rho_\delta(x,y))\,dxdydt,
\end{equation}
as $n \to \infty$, since $H^\va_\Omega(x,y)(1-\rho_\delta(x,y))$ is a legitimate test function for the convergence of $\chi\oo^n \otimes \chi\oo^n$ to $\chi\oo^\infty \otimes \chi\oo^\infty$ weak-$\ast$-$L^\infty(0,T;\mathcal{BM}(\Omega \times \Omega))$.

The measures $\chi\oo^\infty (t,\cdot)$ are weak-$\ast$ limits of {\em continuous measures}, i.e., measures with no atomic parts. Because norms are weak-$\ast$ lower semicontinuous, we find
\begin{equation}\label{wklowsemicontnorms}
\int_{B(x;\delta)} |\chi\oo^\infty(t,y)|\,dy \leq \liminf_{n\to\infty} \int_{B(x;\delta)} |\chi\oo^n(t,y)|\,dy.
\end{equation}
Here we are abusing notation, writing $\ds{\int_{B(x;\delta)}} |\chi\oo^\infty(t,y)|\,dy$ for $\ds{\int_{B(x;\delta)}} d|\chi\oo^\infty(t,y)|$.

Therefore, from Assumption (3) and \eqref{wklowsemicontnorms}, we obtain
\[\int_0^T \left(\sup_{x\in \supp \chi} \int_{B(x;\delta)} |\chi\oo^\infty(t,y)|\,dy \right)\, dt \to 0,\]
as $\delta \to 0$.
An argument similar to what was used to study $A^n_1$ now gives
\begin{equation} \label{limdeltaA2n}
\lim_{\delta\to 0} \lim_{n\to \infty} A^n_2= A^\infty,
\end{equation}
where the integral in $A^\infty$ is to be interpreted as integration against a continuous measure, namely $\chi\oo^\infty \otimes \chi\oo^\infty$.

It follows that
\begin{equation} \label{An}
A^n \to A^\infty \text{ as } n\to\infty.
\end{equation}

Putting together \eqref{An} and \eqref{Bn} we deduce
\begin{equation} \label{AnplusBn}
A^n + B^n \to A^\infty+B^\infty,
\end{equation}
as $n \to \infty$. This establishes \eqref{nonlinterm}.

It follows from \eqref{firstweakvortnuNS}, \eqref{timederiv}, \eqref{laplacian} and \eqref{nonlinterm} that $\oo^\infty$ satisfies the interior weak vorticity formulation \eqref{intwkvortid}. Since we already know that $\oo^\infty \in L^\infty((0,T);H^{-1}(\Omega))\cap L^\infty((0,T);\mathcal{BM}_{\loc}(\Omega))$, we have established that $\oo^\infty$ is an interior weak solution of the incompressible 2D Euler equations, in the sense of Definition \ref{intwkvort}.

As $\bu^\infty=K_\Omega[\oo^\infty]$, the proof of the theorem is concluded once we establish Lemma \ref{weakvortweakvel}.
\end{proof}

Now we give the proof of Lemma \ref{weakvortweakvel}, which is a result on the equivalence between the weak velocity formulation and the interior weak vorticity formulation. The argument is based on the proofs of equivalence contained in \cite{ILNInPrep} and \cite{ILNSu}, with variations due to the fact that the vorticity is only a bounded Radon measure \emph{locally, in the interior of the fluid domain}. When regarded as a distribution in the entire fluid domain, the best regularity for $\oo^\infty$ is $H^{-1}$, which is the same as for wild solutions.

\begin{proof}[Proof of Lemma \ref{weakvortweakvel}]
Let us first assume that $\oo^\infty \in L^\infty((0,T);\mathcal{BM}_{\loc} (\Omega) \cap H^{-1}(\Omega))$ is an interior weak solution in the sense of Definition \ref{intwkvort}. Set $\bu^\infty \equiv K_\Omega[\oo^\infty]$. Now, $K_\Omega=\nabla^\perp(\Delta_0)^{-1}$ and therefore the operator $K_\Omega[\cdot]$ is continuous from $H^{-1}(\Omega)$ to $L^2(\Omega)$. It follows that $\bu^\infty \in L^\infty((0,T);L^2(\Omega))$. We also have from the definition of $K_\Omega$ that $\dv \bu^\infty = 0$ and $\curl \bu^\infty \equiv \nabla^\perp \cdot \bu^\infty = \oo^\infty$ in the sense of distributions, and that the trace of the normal component of $\bu^\infty$ vanishes at $\partial\Omega$, see also Remark \ref{tanganddivfree}.

Let $\Phi \in C^\infty_c(0,T;\Omega)$ with $\dv \Phi (t,\cdot) = 0$. Then $\Phi = \nabla^\perp \va$ for some $\va \in C^\infty(0,T;\Omega)$, and $\va(t,\cdot)$ is constant in a neighborhood of $\partial \Omega$.  Since $\Omega$ is connected and simply connected we may assume without loss of generality that this constant is $0$, so that
$\va   \in C^\infty_c((0,T)\times\Omega)$. From the relation between $\bu^\infty$ and $\oo^\infty$ we obtain that
\begin{equation} \label{tempterm}
\int_0^T \int_\Omega \partial_t \Phi \cdot \bu^\infty \, dxdt = - \int_0^T\int_{\Omega} \partial_t \va (t,x) \oo^\infty (t,x)\, dxdt.
\end{equation}

Let $\chi=\chi(x) \in C^\infty_c(\Omega)$, such that $0\leq \chi \leq 1$ and $\chi \equiv 1$ in a neighborhood of the support of $\va$. In order to establish the weak formulation \eqref{wkvel} for $\bu^\infty$, in view of \eqref{intwkvortid} and \eqref{tempterm}, it remains only to show that
\begin{align}\label{nonlintermy}
\nonumber & \int_0^T\int_\Omega \bu^\infty \otimes \bu^\infty : \nabla \Phi \, dxdt = \\
& \quad \quad - \int_0^T\int_\Omega\int_\Omega H_\Omega^\va  (t,x,y)\chi(x) \oo^\infty(t,x)\chi(y) \oo^\infty(t,y) \, dxdydt\\
\nonumber & \quad \quad - \int_0^T\int_\Omega \int_\Omega K_\Omega(x,y)(1-\chi(y))\chi(x)\cdot\nabla\va(t,x)\oo^\infty(t,x)\oo^\infty(t,y)\,dxdydt.
\end{align}

In the identity above $t$ is merely a parameter, so we  freeze time and show that
\begin{align} \label{freezetime}
\nonumber & \int_\Omega \bu^\infty \otimes \bu^\infty : \nabla \Phi \, dx = \\
& \quad \quad - \int_\Omega\int_\Omega H_\Omega^\va  (t,x,y)\chi(x) \oo^\infty(t,x)\chi(y) \oo^\infty(t,y) \, dxdy\\
\nonumber & \quad \quad - \int_\Omega \int_\Omega K_\Omega(x,y)(1-\chi(y))\chi(x)\cdot\nabla\va(t,x)\oo^\infty(t,x)\oo^\infty(t,y)\,dxdy.
\end{align}
Because time is frozen at $t$ we omit it hereafter.

We adapt what was done in \cite[Theorem 6.1 and Proposition 6.2]{ILNSu}, see also \cite[Theorem 3.4 and Proposition 3.5]{ILNInPrep}, to the situation we have, where only interior estimates are available. Let $\rho_k$ be a cutoff away from the boundary. More precisely, we assume that
\[
\rho_k\in C_c^\infty(\Omega;[0,1]), \quad
\rho_k\equiv1 \text{ in }\Sigma_{\frac2k}^c,
\quad \rho_k\equiv0 \text{ in }\Sigma_{\frac1k},
\quad  \|\nabla\rho_k\|_{L^\infty(\Sigma_{\frac2k}\setminus\Sigma_{\frac1k})}\leq Ck,
\]
where
\begin{equation*}
\Sigma_a=\{x\in\Omega\ ;\ \mathrm{dist}(x,\partial\Omega)\leq a\}.
\end{equation*}

In addition, we introduce
\begin{equation*}
\zeta\in C^\infty_c(\R^2;\R_+)  \mbox{, $\zeta$ is even},\quad \supp\zeta\subset B(0;1/2),\quad \int\zeta=1,
\end{equation*}
and set
\begin{equation}\label{etak}
  \zeta_k(x)=k^2\zeta(kx).
\end{equation}
Next, set $\oo^k\equiv (\rho_k\oo^\infty)\ast\zeta_k, $ $\bu^k \equiv K_\Omega[\oo^k]$.
We first note that $\bu^k$ and $\oo^k$ are smooth functions, $\oo^k=\curl\bu^k$, and $\bu^k$ is divergence free and tangent to $\partial\Omega$. Since $\Phi=\nabla^\perp\va$ and $\va$ is compactly supported, we obtain, by integration by parts,
\begin{equation} \label{basicid}
\int_\Omega \bu^k (x)\otimes \bu^k (x): \nabla \Phi (x)\, dx = -  \int_\Omega \bu^k(x)\cdot\nabla\va (x)\oo^k(x)\,dx.
\end{equation}

We wish to show that the left-hand-side and right-hand-side of \eqref{basicid} converge, respectively, to the left-hand-side and right-hand-side of \eqref{freezetime}.

We begin by analyzing the left-hand-side of \eqref{basicid}. To this end we claim that $\bu^k \to \bu^\infty$ strongly in $ L^2(\Omega) $. The proof of this claim follows precisely the proof of \cite[Proposition 4.8]{ILNSu}, see also \cite[Proposition 2.10]{ILNInPrep}, once we observe that $\oo^k \rightharpoonup \oo^\infty$ in $\mathcal{D}^\prime(\Omega)$. We give an outline for the convenience
of the reader, omitting some of the details. There are three steps: first it is shown that $\bu^k$ is bounded in $L^2(\Omega)$, then it is established that $\bu^k \rightharpoonup \bu^\infty$ weakly in $L^2(\Omega)$. Finally, it is proved that $\|\bu^k\|_{L^2(\Omega)} \to \|\bu^\infty\|_{L^2(\Omega)}$. To show the first step consider $\blds{F}\in C^\infty(\overline{\Omega})$, and set $f\equiv G_\Omega[\curl \blds{F}]$, where $\curl \blds{F} \equiv \nabla^\perp \cdot \blds{F}$. Then $f\in C^\infty(\overline{\Omega})$, $f$ is bounded and vanishes at $\partial\Omega$. Step 1 follows from the observation that
\begin{equation} \label{dragosmagic}
\int_\Omega \bu^k \cdot \blds{F} = - \int_\Omega \oo^k f =  \int_\Omega ( \bu^\infty \cdot \nabla^\perp \rho_k ) \, (f\ast {\zeta_k}) + \int_\Omega \rho_k\bu^\infty\cdot(\nabla^\perp f \ast {\zeta_k}) .
\end{equation}
Using the Hardy inequality it follows that
\[\left| \int_\Omega \bu^k \cdot \blds{F} \right| \leq C \|\bu^\infty\|_{L^2(\Omega)}\|\blds{F}\|_{L^2(\Omega)}.\]
Step 2 holds by virtue of the convergence $\oo^k \rightharpoonup \oo^\infty$ in $\mathcal{D}^\prime(\Omega)$, which is classical. To prove Step 3, $\blds{F}=\bu^k$ is used in \eqref{dragosmagic}, so that
\[
\|\bu^k\|_{L^2(\Omega)}^2 =  \int_\Omega ( \bu^\infty \cdot \nabla^\perp \rho_k ) \, (G_\Omega [\oo^k]\ast {\zeta_k}) + \int_\Omega \rho_k\bu^\infty\cdot(\bu^k \ast  {\zeta_k}).
\]
The first term vanishes as $k\to\infty$ and the second term converges to $\|\bu^\infty\|_{L^2(\Omega)}^2$.

In view of the strong convergence of $\bu^k$ to $\bu^\infty$ it follows that the left-hand-side of \eqref{basicid} converges to the left-hand-side of \eqref{freezetime}.

Now we discuss the right-hand-side of \eqref{basicid}. Let $\chi$ be a cutoff for the support of $\va$ as in the proof of Theorem \ref{mainthm} and set
\[\eta\equiv \mathrm{dist}\{\supp{(1-\chi)}, \supp{\va}\} > 0.\]

We  decompose $\oo^k$ into an interior part and a boundary part:
\[\oo^k_I\equiv (\rho_k\chi\oo^\infty)\ast \zeta_k, \quad  \oo^k_B\equiv\oo^k - \oo^k_I =  [\rho_k(1-\chi)\oo^\infty)]\ast \zeta_k.\]
Correspondingly, we decompose the velocities $\bu^k$:
\[\bu^k_I\equiv K_\Omega [\oo^k_I], \quad  \bu^k_B\equiv \bu^k - \bu^k_I = K_\Omega[\oo^k_B].\]

With this notation we rewrite the right-hand-side of \eqref{basicid} as
\begin{align} \label{intxbdry}
\nonumber &   \int_\Omega \bu^k(x)\cdot\nabla\va (x)\oo^k(x)\,dx  \\
 &\quad \quad =  \int_\Omega \bu^k_I(x)\cdot\nabla\va (x)\oo^k_I(x)\,dx  +  \int_\Omega \bu^k_B(x)\cdot\nabla\va (x)\oo^k_I(x)\,dx \\
\nonumber &\quad \quad    + \int_\Omega \bu^k_I(x)\cdot\nabla\va (x)\oo^k_B(x)\,dx  + \int_\Omega \bu^k_B(x)\cdot\nabla\va (x)\oo^k_B(x)\,dx.
\end{align}

We claim that for sufficiently large $k$ the two integrals in the last line above vanish. To see this, first recall that the support of $1-\chi$ is at a distance $\eta > 0$ from the support of $\va$. Let $k > 1/\eta$. Then, by construction, if $x \in \supp{\va}$ and $y \in \Omega$ is such that $|x-y|< \eta/2$, it follows that $\oo^k_B(y)=0$. In particular, if $k> 1/\eta$, $\nabla\va \,\oo^k_B \equiv 0$ on all of $\Omega$.

Next we analyze the second integral on the right-hand-side of \eqref{intxbdry}:
\begin{equation} \label{secondintegral}
  \int_\Omega \bu^k_B(x)\cdot\nabla\va (x)\oo^k_I(x)\,dx = \int_\Omega \int_\Omega K_\Omega(x,y)\cdot\nabla \va(x) \oo^k_I(x)\oo^k_B(y)\,dxdy.
\end{equation}
As noted above, if $k > 1/\eta$ then $\mathrm{dist}\{\supp{\va}, \,\supp{\oo^k_B}\} \geq \eta/2$. Let $\psi=\psi(z)$ be a cut-off of $|z|\geq \eta/2$, so that $\psi \in C^\infty (\real^2)$, $0\leq \psi \leq 1$, $\psi(z) \equiv 1$ if $|z|\geq \eta/2$ and $\psi(z) \equiv 0$ if $|z|<\eta/4$. Then, for $k > 1/\eta$, we may re-write \eqref{secondintegral} as
\begin{equation} \label{secondintegralprime}
  \int_\Omega \bu^k_B(x)\cdot\nabla\va (x)\oo^k_I(x)\,dx = \int_\Omega \int_\Omega K_\Omega(x,y)\cdot\nabla \va(x) \psi(x-y) \oo^k_I(x)\oo^k_B(y)\,dxdy.
\end{equation}
Now, arguing similarly to what was done in the proof of Theorem \ref{mainthm}, we obtain $K_\Omega (x,y)\cdot\nabla\va (x)\psi(x-y)\in H^1_0(\Omega\times\Omega)$. We easily deduce that $\oo^k_I \otimes \oo^k_B \rightharpoonup \chi\oo^\infty \otimes (1-\chi)\oo^\infty$ weakly in $H^{-1}(\Omega \times \Omega)$. Therefore, the sequence of integrals on the left-hand-side of \eqref{secondintegral} converge to the second integral on the right-hand-side of \eqref{freezetime}.

Finally, we deal with the first integral on the right-hand-side of \eqref{intxbdry}. We have:
\begin{equation} \label{firstintegral}
\int_\Omega \bu^k_I(x)\cdot\nabla\va (x)\oo^k_I(x)\,dx = \int_\Omega \int_\Omega H_\Omega^\va (x,y) \oo^k_I(x)\oo^k_I(y)\,dxdy.
\end{equation}

Now, $\chi \oo^\infty \in \mathcal{BM} \cap H^{-1}(\Omega)$,  hence the result in \cite[Proposition 4.8]{ILNSu}, see also \cite[Proposition 2.10]{ILNInPrep}, applies to $\oo^k_I$:
\begin{align} \label{ookIconv}
\nonumber & \oo^k_I \mbox{ is bounded in } L^1(\Omega), \quad \oo^k_I \rightharpoonup \chi\oo^\infty \mbox{ weak}-\ast \mathcal{BM}(\Omega) \mbox{ and } \\
& \mbox{any weak$-\ast$ limit in $\mathcal{BM}(\Omega)$,  $\mu$, of $|\oo^k_I|$ is a continuous measure,} \\
\nonumber & i.e., \mu (P) = 0, \mbox{ for any } P \in \Omega.
\end{align}
For convenience of the reader we give an outline of the proof of \eqref{ookIconv}; more details can be found in \cite{ILNSu,ILNInPrep}. The statements on the first line of \eqref{ookIconv} follow from the definition of $\oo^k_I$. Let us discuss the last statement in \eqref{ookIconv}, regarding  weak$-\ast$ limits in $\mathcal{BM}(\Omega)$ of $|\oo^k_I|$.
 Let $\chi\oo^\infty = \nu^+ + \nu^-$ be an orthogonal (Hahn-Jordan) decomposition into positive and negative parts, with disjoint supports. Since $\chi \oo^\infty \in H^{-1}(\Omega)$ is a measure, it is necessarily a continuous measure. Since the supports of $\nu^+$ and $\nu^-$ are disjoint, they too are each continuous measures.   Write  $\oo^k_I = (\rho_k \nu^+)\ast\zeta_k + (\rho_k \nu^-)\ast\zeta_k$; this is a decomposition into positive and negative measures, albeit with supports that are no longer necessarily disjoint. Still, we have
 \[|\oo^k_I| \leq (\rho_k \nu^+) \ast \zeta_k - (\rho_k \nu^-)\ast\zeta_k.
 \]
In addition, $(\rho_k \nu^{\pm})\ast\zeta_k \rightharpoonup \nu^{\pm}$ weak-$\ast$ in $\mathcal{BM}(\Omega)$.
 Now, let $\mu$ be a weak-$\ast$ $\mathcal{BM}(\Omega)$ limit of $|\oo^k_I|$. Then we find $\mu$ is a nonnegative measure which is bounded in the sense of measures by $\nu^+-\nu^-\equiv |\chi\oo^\infty|$. Therefore, $0\leq \mu(P) \leq |\chi\oo^\infty|(P)=0$ for any $P \in \Omega$, as desired.

Next we recall \cite[Lemma 6.3.1]{C1998}, where it was established that, if $\nu_k \rightharpoonup \nu$ weak-$\ast$ $\mathcal{BM}$, then $\int f \nu_k \to \int f \nu$ for any bounded test function $f$ which is continuous off of a $\mu$-negligible set, where $\mu$ is any weak-$\ast$ limit of $|\nu_k|$. We may use this result with $f = H_\Omega^\va$ and $\nu_k = \oo^k_I\otimes\oo^k_I$, since it has already been observed, in the proof of Theorem \ref{mainthm}, that $H_\Omega^\va$ is continuous off of the diagonal of $\overline\Omega \times \overline \Omega$, and we have established in \eqref{ookIconv}, that the diagonal is a negligible set for any weak-$\ast$ limit of $|\oo^k_I\otimes\oo^k_I|$, together with the fact that $\oo^k_I\otimes\oo^k_I \rightharpoonup \chi\oo^\infty\otimes\chi\oo^\infty$. We conclude that the integrals on the left-hand-side of \eqref{firstintegral} converge to the first integral on the right-hand-side of \eqref{freezetime}.

Putting together our analysis of the terms in \eqref{intxbdry} we obtain that the right-hand-side of \eqref{basicid} converges to the right-hand-side of \eqref{freezetime}. This establishes \eqref{nonlintermy}.

Conversely, suppose that $\bu^\infty \in L^\infty((0,T);L^2(\Omega))$ is a weak solution of the Euler equations as in Definition \ref{weakEuler}, and assume further that $\oo^\infty\equiv\curl\bu^\infty \in L^\infty((0,T);\mathcal{BM}_{\loc}(\Omega)\cap H^{-1}(\Omega))$. Let $\va \in C^\infty_c((0,T)\times\Omega)$ and set $\Phi\equiv\nabla^\perp\va$. It is easy to see that $\Phi \in C^\infty_c((0,T)\times\Omega)$, $\dv \Phi = 0$ and \eqref{tempterm} holds true. In addition, in view of the regularity assumption on $\oo^\infty$, the proof we gave of \eqref{nonlintermy}, for a suitable cut-off function $\chi$, may be used once again. Putting together \eqref{tempterm}, \eqref{nonlintermy}  and \eqref{wkvel} yields that $\oo^\infty$ satisfies \eqref{intwkvortid}. This concludes the proof.
\end{proof}

Next we give an example of a sequence of solutions of the Navier-Stokes equations which satisfy the hypothesis of Theorem \ref{mainthm} and for which the limiting Euler solution is not smooth. In fact, it is a vortex sheet and thus it falls outside the scope of the Kato criterion.

We take $\Omega = \{x\in\real^2 \;|\; |x|<1\}$, the unit disk. Set
\begin{equation} \label{u0vortsheet}
\bu_0 =
\left\{
\begin{array}{ll}
0, & \text{ if } |x|< \frac{1}{2},\\
\ds{\frac{x^\perp}{|x|^2}}, & \text{ if } \frac{1}{2} < |x| < 1.
\end{array}
\right.
\end{equation}
The corresponding vorticity is
\begin{equation} \label{omega0vortsheet}
\oo_0 = \delta_{\{|x|=1/2\}}.
\end{equation}

Let $\nu_n$ be any sequence of positive numbers such that $\nu_n \to 0$.  Choose circularly symmetric approximations $\bu_0^n \in C^\infty (\Omega) \cap L^2(\Omega)$, $\dv \bu^n_0=0$, $\bu^n_0\cdot\n_{\big|_{\{|x|=1\}}}=0$, such that
\[\bu_0^n \to \bu_0 \text{ strongly in } L^2(\Omega) \text{ and, for } \oo_0^n=\nabla^\perp\cdot\bu_0^n, \text{ it holds that}\]
\[ \oo_0^n \geq 0, \quad \oo_0^n \text{ circularly symmetric, }
\text{ and } \|\oo_0^n \|_{L^1(\Omega)} \leq \|\oo_0\|_{\mathcal{BM}(\Omega)} \equiv \pi.\]

Let $\bu^n$ be the unique solution of the initial boundary value problem for the Navier-Stokes equations \eqref{nuNS} with viscosity $\nu_n$ and initial velocity $\bu_0^n$ above, as in Theorem \ref{mainthm}. Let $\oo^n = \nabla^\perp \cdot \bu^n$.

\begin{prop} \label{vortsheetexample}
The sequences $\{\bu^n\}$, $\{\oo^n\}$ satisfy Assumptions (1), (2) and (3) of Theorem \ref{mainthm}. Furthermore, the weak limit of $\bu^n$, $\bu^\infty$, is time-independent and equal to $\bu_0$.
\end{prop}

\begin{proof}[Proof of Proposition~\ref{vortsheetexample}]
We first note that $\bu_0 \in L^2(\Omega)$ is circularly symmetric. Let $\blds{v}^n$ denote the solution of \eqref{nuNS} with viscosity $\nu_n$ and initial data $\bu_0$. It follows from the analysis developed in \cite{LMN2008} that
\[\blds{v}^n \rightarrow \bu_0, \text{ strongly in } L^\infty((0,T);L^2(\Omega)).\]
In addition, $\bu^n-\blds{v}^n$ is the solution of \eqref{nuNS} with viscosity $\nu_n$ and initial data $\bu_0^n-\bu_0$. The most elementary of energy estimates yields that
\[
\sup_{t \in [0,T]} \|\bu^n-\blds{v}^n\|_{L^2(\Omega)}\leq \|\bu_0^n-\bu_0\|_{L^2(\Omega)} \rightarrow 0 \text{ as } n \to \infty.\]
Hence $\{\bu^n\}$ satisfies Assumption (1) with $\bu^\infty=\bu_0$.

Next we recall the result in \cite[Proposition 9.4]{LMNT2008}, where it was established that
\[\|\oo^n\|_{L^\infty((0,T);L^1(\Omega))} \leq 4\|\oo^n_0\|_{L^1(\Omega)}.\]
Therefore we find, from the construction of $\bu_0^n$,
\[\|\oo^n\|_{L^\infty((0,T);L^1(\Omega))} \leq 4\pi,\]
which gives Assumption (2).

Lastly, we establish the uniform decay of the vorticity maximal function. We first show that $\oo^n$ locally may be written as the sum of a positive measure in $H^{-1}(\Omega)$ and a uniformly bounded function. To see this, fix $\mathcal{K}\subset\subset\Omega$ and let $\vare>0$ satisfy  $\mathcal{K}\subset\{|x|< 1-3\vare\}$. Choose $\chi_\vare \in C^\infty_c(\Omega)$ such that $0\leq\chi_\vare\leq 1$, $\chi_\vare (x) \equiv 1 $ if $|x|< 1-2\vare$, $\chi_\vare (x) \equiv 0$ if $1-\vare<|x|\leq 1$. Let $\oo^n_\vare \equiv \chi_\vare\oo^n$, and note that $\oo^n_\vare \equiv \oo^n$ on $\mathcal{K}$. We extend $\oo^n_\vare$ to all of $\real^2$ by setting it to vanish outside of $\Omega$. Observe that $\oo^n_\vare$ is a solution of the following heat equation in the full plane:
\[\partial_t\oo^n_\vare = \nu_n \Delta \oo^n_\vare - \nu_n \oo^n_\vare \Delta\chi_\vare-2\nu_n \nabla \oo^n \cdot \nabla \chi^n_\vare \equiv \nu_n\Delta\oo^n_\vare - F^n_\vare.\]
It follows that
\[ \oo^n_\vare = e^{\nu_n t\Delta}[\oo^n_\vare(0,\cdot)] - \int_0^t e^{\nu_n(t-s)\Delta}[F^n_\vare(s,\cdot)]\,ds \equiv I_n + I\!\!I_n.\]
We will analyze $(\oo^n_\vare)_{\big|_{\mathcal{K}}}\equiv \oo^n_{\big|_{\mathcal{K}}}$.
We make two claims.
\begin{claim} \label{Claim1}
We have $\{I_n = e^{\nu_n t \Delta}[\oo^n_\vare(0,\cdot)]\}_n$ is a bounded subset of $L^\infty((0,T);\mathcal{BM}(\real^2)\cap H^{-1}(\real^2))$ and $e^{\nu_nt\Delta}[\oo^n_\vare(0,\cdot)]\geq 0$.
\end{claim}

\begin{claim} \label{Claim2}
There exists $C=C(\|\oo_0\|_{\mathcal{BM}(\Omega)},\vare)>0$ such that, on $\mathcal{K}$,
\[\left|I\!\!I_n = \int_0^t e^{\nu_n(t-s)\Delta}[F^n_\vare(s,\cdot)]\,ds\right|\leq C,\]
for all $n$, $t\in [0,T]$.
\end{claim}

It follows from Claim \ref{Claim1} and well-known estimates in potential theory, see also \cite{Schochet1995}, that
\[\int_{B(x;r)\cap \Omega} I_n \, dy \leq \left( \sup_n \|I_n\|_{H^{-1}(\Omega)} \right) |\log r |^{-1/2} \to 0,\]
as $r\to 0$, uniformly in $n$ and $x \in \mathcal{K}$. In addition, from Claim \ref{Claim2} we find
\[\int_{B(x;r)\cap \Omega} | I\!\!I_n | \, dy \leq Cr^2 \to 0,\]
as $r\to 0$, uniformly in $n$ and $x \in \mathcal{K}$. Thus Claims \ref{Claim1} and \ref{Claim2} are sufficient to establish Assumption (3) of Theorem \ref{mainthm}.

\begin{proof}[Proof of Claim \ref{Claim1}]
Recall $\oo^n_\vare(0,\cdot)=\chi_\vare \oo^n_0$. By hypothesis, $\oo^n_0\geq 0$ and, also, $\chi_\vare \geq 0$. It follows from the maximum principle for the heat equation that
$e^{\nu_n t \Delta}[\oo^n_\vare(0,\cdot)]\}_n\geq 0$. In addition, the $L^1$-norm of $e^{\nu_n t \Delta}[\oo^n_\vare(0,\cdot)]\}_n$ is a non-increasing function of time, so that, using again the hypotheses we made on $\oo^n_0$,
\[\|e^{\nu_n t \Delta}[\oo^n_\vare(0,\cdot)] \|_{L^1(\real^2)}\leq \|\oo^n_\vare(0,\cdot)\|_{L^1} \leq \|\oo_0\|_{\mathcal{BM}(\Omega)}.\]
Lastly, we argue that $e^{\nu_n t \Delta}[\oo^n_\vare(0,\cdot)]$ is uniformly bounded in $L^\infty((0,T);H^{-1}(\real^2))$. We first observe that
\[\oo^n_\vare(0,\cdot)=\chi_\vare \oo^n_0 = \nabla^\perp \cdot (\chi_\vare \bu^n_0)-\bu^n_0\cdot\nabla^\perp\chi_\vare.\]
Therefore,
\[e^{\nu_n t \Delta}[\oo^n_\vare(0,\cdot)] = \nabla^\perp \cdot \left\{ e^{\nu_n t \Delta}[\chi_\vare \bu^n_0]\right\} - e^{\nu_n t \Delta}[\bu^n_0\cdot\nabla^\perp\chi_\vare].
\]
Clearly, $e^{\nu_n t \Delta}[\chi_\vare \bu^n_0]$ and $e^{\nu_n t \Delta}[\bu^n_0\cdot\nabla^\perp\chi_\vare]$ are, both, bounded subsets of $L^\infty((0,T);L^2(\real^2))$, hence the desired assertion follows.
\end{proof}

\begin{proof}[Proof of Claim \ref{Claim2}]
We write
\begin{align*}
-\int_0^t e^{\nu_n(t-s)\Delta}[F^n_\vare(s,\cdot)]\,ds
& =
\nu_n\int_0^t e^{\nu_n (t-s)\Delta}[\oo^n \Delta\chi_\vare (s,\cdot)]\,ds +
2\nu_n \int_0^t e^{\nu_n (t-s)\Delta}[\nabla \oo^n \cdot \nabla \chi^n_\vare (s,\cdot)]\,ds\\
& \equiv A^n+B^n.
\end{align*}
Now,
\begin{align*}
|A^n {(t,x)}|
&=\left|
\nu_n\int_0^t \int_{\real^2} \frac{1}{4\pi\nu_n(t-s)}e^{-|x-y|^2/(4\nu_n (t-s))}\oo^n(s,y) \Delta\chi_\vare (y) \,dyds
\right|\\
&\leq \frac{\nu_n}{\pi\vare^2}\sup_{\rho} (\rho e^{-\rho})T\|\oo^n\|_{L^\infty((0,T);L^1(\Omega))}\|\Delta\chi_\vare\|_{L^\infty(\real^2)},
\end{align*}
since $\supp \Delta \chi_\vare \subset \{1-2\vare \leq |y| \leq 1-\vare\}$ and $|x|<1-3\vare$. It follows that $|A^n|$ is bounded, uniformly with respect to $n$, $t\in (0,T)$, for $x\in\mathcal{K}$.

Next, we examine $B^n$:
\begin{align*}
|B^n {(t,x)}|
&= \left| 2\nu_n \int_0^t \int_{\real^2} \frac{1}{4\pi\nu_n(t-s)}e^{-|x-y|^2/(4\nu_n (t-s))}\nabla \oo^n (s,y)\cdot \nabla \chi^n_\vare(y)\,dyds
\right| \\
&= \left|
-2\nu_n \int_0^t \int_{\real^2} \frac{1}{4\pi\nu_n(t-s)}e^{-|x-y|^2/(4\nu_n (t-s))} \oo^n (s,y)\Delta \chi^n_\vare(y)\,dyds\right. \\
&\qquad  \left.
+4\nu_n \int_0^t \int_{\real^2} \frac{(x-y)}{\pi(4\nu_n(t-s))^2}e^{-|x-y|^2/(4\nu_n (t-s))} \oo^n (s,y)\cdot \nabla \chi^n_\vare(y)\,dyds
\right|.
\end{align*}
Therefore, reasoning similarly as for $A^n$, we find
\begin{align*}
|B^n|
&\leq \frac{\nu_n}{\pi\vare^2}\sup_{\rho} (\rho e^{-\rho})T\|\oo^n\|_{L^\infty((0,T);L^1(\Omega))}\|\Delta\chi_\vare\|_{L^\infty(\real^2)}\\
&\qquad +
\frac{\nu_n}{\pi\vare^3}\sup_{\rho} (\rho^2 e^{-\rho})T\|\oo^n\|_{L^\infty((0,T);L^1(\Omega))}\|\nabla\chi_\vare\|_{L^\infty(\real^2)}.
\end{align*}
Hence it follows that $|B^n|$ is uniformly bounded, for $x \in \mathcal{K}$, with respect to $n$ and $t\in (0,T)$.
\end{proof}

This concludes the proof of Proposition \ref{vortsheetexample}.
\end{proof}

\ \par \

\textbf{Acknowledgments.}  M.C. Lopes Filho acknowledges the support of Conselho Nacional de Desenvolvimento Cient\'{\i}fico e Tecnol\'ogico -- CNPq through grant \# 306886/2014-6 and of FAPERJ through grant \# E-26/202.999/2017 . H.J. Nussenzveig Lopes thanks the support of Conselho Nacional de Desenvolvimento Cient\'{\i}fico e Tecnol\'ogico -- CNPq through grant \# 307918/2014-9 and of FAPERJ through grant \# E-26/202.950/2015. V. Vicol was partially supported by the NSF grant DMS-1652134 and an Alfred P. Sloan Fellowship. P. Constantin was  partially supported by the NSF grant DMS-1713985.

\footnotesize{
\bibliographystyle{plain}
%\bibliography{Refs}
\def\cprime{$'$} \def\polhk#1{\setbox0=\hbox{#1}{\ooalign{\hidewidth
  \lower1.5ex\hbox{`}\hidewidth\crcr\unhbox0}}}

}

\end{document}